\documentclass{amsart}
\usepackage{amssymb,amsmath, amsfonts, amsrefs, tikz, epsfig, float}
\usepackage{amsthm}
\usepackage{mathrsfs}
\usepackage{enumitem, indentfirst}
\usepackage[fleqn,tbtags]{mathtools}
\usepackage{color}
 \newtheorem{theorem}{Theorem}[section]
 \newtheorem{corollary}[theorem]{Corollary}
 \newtheorem{lemma}[theorem]{Lemma}
 \newtheorem{proposition}[theorem]{Proposition}

 \theoremstyle{definition}

 \theoremstyle{remark}
 \newtheorem{remark}[theorem]{Remark}
  \numberwithin{equation}{section}

\renewcommand{\phi}{\varphi}
\renewcommand{\theta}{\vartheta}

\DeclareMathOperator{\tform}{\mathfrak{t}}

\DeclareMathOperator{\wform}{\mathfrak{w}}

\DeclareMathOperator{\mul}{mul}
\DeclareMathOperator{\re}{Re}

\DeclarePairedDelimiterX\sipt[2]{(}{)^{}_{T}}{#1\,\delimsize\vert\,#2}
\DeclarePairedDelimiterX\sipv[2]{(}{)_{v}}{#1\,\delimsize\vert\,#2}
\DeclarePairedDelimiterX\sipw[2]{(}{)_{w}}{#1\,\delimsize\vert\,#2}

\newcommand{\balg}{\mathscr{B}}

\newcommand{\dom}{\operatorname{dom}}

\newcommand{\ran}{\operatorname{ran}}

\newcommand{\pt}{P^{}_T}
\newcommand{\qt}{Q^{}_T}

\newcommand{\hil}{\mathcal H}
\newcommand{\gil}{\mathcal G}

\newcommand{\kil}{\mathcal K}

\DeclareMathOperator{\reg}{s}
\DeclareMathOperator{\sing}{sing}
\newcommand{\treg}{T_{\reg}}

\DeclarePairedDelimiterX\sip[2]{(}{)}{#1\,\delimsize\vert\,#2}
\DeclarePairedDelimiterX\siptilde[2]{(}{)_{\!_{\widetilde{A}}}}{#1\,\delimsize\vert\,#2}
\DeclarePairedDelimiterX\sipf[2]{(}{)_{f}}{#1\,\delimsize\vert\,#2}
\DeclarePairedDelimiterX\sipg[2]{(}{)_{g}}{#1\,\delimsize\vert\,#2}
\DeclarePairedDelimiterX\siptw[2]{(}{)_{\tform+\wform}}{#1\,\delimsize\vert\,#2}
\DeclarePairedDelimiterX\set[2]{\{}{\}}{#1\,:\,#2}
\DeclarePairedDelimiterX\dual[2]{\langle}{\rangle}{#1,#2}
\DeclarePairedDelimiterX\sipa[2]{(}{)_{\!_A}}{#1\,\delimsize\vert\,#2}
\DeclarePairedDelimiterX\sipc[2]{(}{)_{\!_C}}{#1\,\delimsize\vert\,#2}
\DeclarePairedDelimiterX\sipab[2]{(}{)_{\!_{A+B}}}{#1\,\delimsize\vert\,#2}
\DeclarePairedDelimiterX\sipb[2]{(}{)_{\!_B}}{#1\,\delimsize\vert\,#2}

\newcommand{\opmatrix}[4]{\begin{bmatrix} #1 &  #2 \\   #3& #4\end{bmatrix}}
\newcommand{\pair}[2]{\begin{bmatrix}#1 \\  #2 \end{bmatrix}}
\allowdisplaybreaks
\title[Canonical graph contractions]{Canonical graph contractions of linear relations on Hilbert spaces}

\author[Zs. Tarcsay]{Zsigmond Tarcsay}
\thanks{The corresponding author Zs. Tarcsay was supported by DAAD-TEMPUS Cooperation Project ``Harmonic Analysis and Extremal Problems'' (grant no. 308015). Project no. ED 18-1-2019-0030 (Application-specific highly reliable IT
solutions)
has been implemented with the support provided from the National Research,
Development and Innovation Fund of Hungary, financed under the Thematic
Excellence
Programme funding scheme.}
\address{%
Zs. Tarcsay \\ Department of Applied Analysis  and Computational Mathematics\\ E\"otv\"os Lor\'and University\\ P\'azm\'any P\'eter s\'et\'any 1/c.\\ Budapest H-1117\\ Hungary}
\email{tarcsay@cs.elte.hu}

\author[Z. Sebesty\'en]{Zolt\'an Sebesty\'en}
\address{%
Z. Sebesty\'en \\ Department of Applied Analysis  and Computational Mathematics\\ E\"otv\"os Lor\'and University\\ P\'azm\'any P\'eter s\'et\'any 1/c.\\ Budapest H-1117\\ Hungary}
\email{sebesty@cs.elte.hu}


\dedicatory{Dedicated to Henk de Snoo on the occasion of his 75th birthday}

\begin{document}
\maketitle
\begin{abstract}
Given a closed linear relation $T$ between two Hilbert spaces $\hil$ and $\kil$, the corresponding first and second coordinate projections $P_T$ and $Q_T$ are both linear contractions from $T$ to $\hil$, and  to $\kil$, respectively. In this paper we investigate the features of these graph contractions.  We show among others that $P_T^{}P_T^*=(I+T^*T)^{-1}$, and that $Q_T^{}Q_T^*=I-(I+TT^*)^{-1}$. The ranges  $\ran P_T^{*}$ and $\ran Q_T^{*}$ are proved to be closely related to the so called `regular part' of $T$. The connection of the graph projections to Stone's decomposition of a closed linear relation is also discussed. 
\end{abstract}
\section{Introduction}
When dealing with (unbounded) operators, it is sometimes beneficial to identify them with their graph, that is, to treat them as linear subspaces of the corresponding product space. This approach is especially useful if the operator in question is non-closable, that is, when the closure of its graph is not the graph of a `single-valued' operator anymore. Similarly, the adjoint of a linear transformation can be interpreted as an operator only if it is densely defined. 

The theory of linear relations (or `multi-valued' linear operators in other words) between Hilbert spaces goes back at least to the fundamental paper by R. Arens \cite{Arens}. By definition, a linear relation $T$ between two Hilbert spaces $\hil$ and $\kil$ is just a vector subspace of the product Hilbert space $\hil\times \kil$. In this way, the only (but significant) difference between operators and relations is that $\{0,k\}\in T$ does not necessarily imply $k=0$. However, this generality greatly simplifies the handling of operations such as taking closure, adjoint, or inverse. 

A linear relation $T$ consists of certain ordered pairs $\{x,y\}$ of $\hil\times\kil$, so one may consider the first and second coordinate projections of $T$ into $\hil$ and $\kil$, respectively:
\begin{equation*}
    \pt\{x,y\}\coloneqq x,\qquad \qt\{x,y\}\coloneqq y,\qquad \{x,y\}\in\hil. 
\end{equation*}
Note that both $\pt$ and $\qt$ are continuous (with norm bound $1$) if we endow $T$ with the inner product coming form that of $\hil\times\kil$. We shall therefore call $\pt$ and $\qt$ the \emph{canonical contractions} of $T$. Assume in addition that $T$ is a closed relation, then the domain of $\pt$ and $\qt$ becomes a Hilbert space and thus we may take the adjoint operators $P_T^*:\hil\to T$ and $Q_T^*:\kil\to T$, and also the product operators $ \pt P_T^*, \pt Q_T^*, \qt P_T^*$ and $ \qt Q_T^* $ are well defined contractions. 

 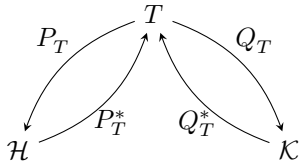
\begin{figure}[H]
\begin{tikzpicture}[->, node distance=1.8cm, >=stealth]
\node (G) {$T$};
\node (K) [right of=G, below of=G] {$\mathcal{K}$};
\node (H) [left of=G, below of=G] {$\mathcal{H}$};
\draw (G) to [bend right=25] node[pos=0.5,above] {$P^{}_T$\quad\quad} (H);
\draw (H) to [bend right=25] node[pos=0.5,below] {~~$P^{*}_T$}  (G);
\draw (G) to [bend left=25] node[pos=0.5,above] {\quad$Q^{}_T$} (K);

\draw (K) to [bend left=25] node[pos=0.5,below] {$Q^{*}_T$~~\quad}(G);
\end{tikzpicture}
\caption{The canonical graph contractions and their adjoint.}
 \end{figure} 

The present paper is devoted to the study of these canonical contractions and their connection with the closed linear relation $T$. First we are going to show that the range $\ran P_T^*\subseteq T$ is always a  regular relation (that is, the graph of a closable operator). Namely, its closure is identical with the regular part of $T$. (Recall that the regular part $\treg$ of $T$ is defined as $\treg\coloneqq (I-P_m)T$ where $P_m$ is the orthogonal projection of $\kil$ onto $\mul T$, see \cite{Hassi2007}.)  It will also turn out that $\ran P_T^*$ is the graph of the restriction of $\treg$ to $\dom T^*T$. 

In \cite{SebTarcs2019} the authors established necessary and sufficient conditions for a pair $S,T$ of operators in order that they satisfy 
\begin{equation}\label{E:symmpair}
S^*=T\qquad  \mbox{and}\qquad  T^*=S,
\end{equation}
cf. also  \cites{SebTarcs2016, SebTarcs2020, Popovici, SebTarcs2013} and \cite{Sandovici} for the relation case. Below we provide some further characterizations for \eqref{E:symmpair} by means of the corresponding graph contractions $\pt,\qt$ and $P_S^{},Q_S^{}$. As an application we offer a new proof of the self-adjointness of the relations $T^*T^{**}$ and $T^{**}T^*$ by proving that  $(I+T^*T^{**})^{-1}=\pt P_T^*$ and $\qt Q_T^*= I-(I+T^{**}T^*)^{-1}$. Finally, we show how Stone's decomposition \cite{Stone} of a closed linear relation $T$ can be obtain by applying the results of the paper. 

\section{Linear relations}
Throughout the paper, $\hil$ and $\kil$ will denote real or complex Hilbert spaces. A linear relation $T$ between $\hil$ and $\kil$ is nothing but a linear subspace of the product Hilbert space $\hil\times \kil$. We shall call the relation $T$  closed if it is a closed subspace of $\hil\times\kil$. Accordingly, the closure $\overline T$ of $T$ is always a closed linear relation, and being so, it becomes a Hilbert space with respect to the induced inner product  
\begin{equation*}
    \sipt{\{x,y\}}{\{u,v\}}\coloneqq \sip{x}{u}^{}_{\hil}+\sip{y}{v}^{}_{\kil},\qquad \{x,y\},\{u,v\}\in \overline T.
\end{equation*}
If we refer to $\overline T$ as the above Hilbert space, we shall denote it by $\gil(T)$. 

Recall that every linear operator $T:\hil\to\kil$ when identified with its graph is a linear relation:
\begin{equation*}
    T\equiv \set{\{x,Tx\}}{x\in\dom T}.
\end{equation*}
Nevertheless, the closure (of the graph) of a linear operator is no longer an operator in general, namely,  it may be that $\{0,k\}\in\overline{T}$ for some non-zero $k$. Accordingly, we call  $T$ closable if its closure $\overline{T}$ is itself an operator.

The domain, range, kernel and multivalued part of a linear relation $T$ are defined to be the following linear subspaces, respectively:
\begin{align*}
    \dom T&\coloneqq \set[\big]{x\in\hil}{\{x,y\}\in T},& \ran T\coloneqq\set[\big]{y\in\kil}{\{x,y\}\in T},\\
    \ker T&\coloneqq \set[\big]{x\in\hil}{\{x,0\}\in T},&\mul T\coloneqq \set[\big]{y\in\kil}{\{0,y\}\in T}.
\end{align*}
It is immediate that $\ker T$ and $\mul T$ are both closed subspaces whenever $T$ itself is closed. It goes also without saying that $T$ is (the graph of) and operator if and only if $\mul T=\{0\}$, and that $T$ is (the graph of) a closable operator if and only of $\mul \overline{T}=\{0\}$. 

The inverse of a linear relation $T$ is  defined as 
\begin{equation*}
    T^{-1}\coloneqq \set[\big]{\{y,x\}}{\{x,y\}\in T}.
\end{equation*}
If $S$ and $T$ are both linear relations then their product $TS$ is given by 
\begin{equation*}
    TS\coloneqq \set[\big]{\{x,z\}}{ \{x,y\}\in S \quad \mbox{and} \quad \{y,z\}\in T \quad \mbox{for some $y$}}.
\end{equation*}
The operatorlike sum of $S$ and $T$ is 
\begin{equation*}
    S+T\coloneqq \set[\big]{\{x,y+z\}}{\{x,y\}\in S, \{x,z\}\in T},
\end{equation*}
just like in the case of operators, 
while the componentwise (or Minkowski) sum is 
\begin{equation*}
    S\,\widehat+\,T\coloneqq \set[\big]{\{x+v,y+z\}+}{\{x,y\}\in S, \{v,z\}\in T}.
\end{equation*}

The adjoint of a linear relation $T$ is defined by 
\begin{equation*}
    T^*\coloneqq W(T)^{\perp},
\end{equation*}
where $W:\hil\times\kil\to\kil\times \hil$ is the `flip' operator 
\begin{equation}\label{E:Wflip}
    W\{h,k\}\coloneqq \{k,-h\},\qquad \{h,k\}\in \hil\times \kil.
\end{equation}
It is immediate that   $T^*$ is a closed linear relation between $\kil$ and $\hil$ and that $T^{**}=\overline{T}$. Note that the following  orthogonal decomposition of $\kil\times\hil$ holds also true:
\begin{equation}
    T^*\,\widehat \oplus\, W(T^{**})=\kil\times \hil. 
\end{equation}
Another equivalent definition of $T^*$ might be given in terms of the inner product, namely,
\begin{equation*}
    \{k,h\}\in T^*\quad\iff\quad \sip{y}{k}^{}_{\kil}=\sip{x}{h}^{}_{\hil},\qquad\forall \{x,y\}\in T.
\end{equation*}
Recall also the following identities:
\begin{equation*}
    \ker T^{*}=(\ran T)^{\perp},\qquad \mul  T^*=(\dom T)^{\perp}.
\end{equation*}
 
 For a given linear relation $T$, let us denote by $P_m$ the orthogonal projection of $\kil$ onto $\mul \overline{T}$. The \textit{regular part} of $T$ is defined as the linear relation 
\begin{equation}\label{E:Treg}
    \treg\coloneqq \set[\big]{\{x,(I-P_m)y\}}{\{x,y\}\in T}.
\end{equation}
It can be shown that  $\treg$ is (the graph of) a  closable operator. In contrast, the  \textit{singular part} following linear relation
\begin{equation}\label{E:Tsing}
    T_{\sing}\coloneqq \set[\big]{\{x,P_my\}}{\{x,y\}\in T}
\end{equation}
is a so called singular relation which means that $\overline{T_{\sing}}$ is the product of two closed subspaces. By means of the regular and singular parts, the linear relation $T$ allows the following canonical sum decomposition
\begin{equation*}
    T=\treg+T_{\sing},
\end{equation*}
 see \cite{Hassi2007}*{Theorem 4.1}. Note also immediately that the regular and singular parts may be written as  
\begin{equation*}
    \treg=(I-P_m)T,\qquad T_{\sing}=P_mT.
\end{equation*}
 We shall also use the fact that ``regular part'' and ``closure'' operations commute in the sense that
\begin{equation}\label{E:tregclosure}
    (\treg)^{**} =(T^{**})_{\reg},
\end{equation}
see \cite{Hassi2007}*{Proposition 4.5}. An important consequence of this result is that  the regular part of a closed linear relation is closed itself, and also that $T_{\reg}\subseteq T$, provided that $T$ is closed.

The interested reader is referred to the books \cites{deSnoobook, Schmudgen} and papers \cites{Hassi2007, Hassi2009,  Arens}  where, in addition to the proofs of the above statements,  more information about linear relations can be found. 

\section{Canonical graph contractions of a linear relation}
Let $T$ be a linear relation between the real or complex Hilbert spaces $\hil$ and $\kil$. The canonical graph contractions $\pt:\overline T\to \hil$ and $\qt:\overline T\to \kil$ of  $T$ are defined as the mappings 
\begin{equation*}
    \pt\{x,y\}\coloneqq x,\qquad \qt\{x,y\}\coloneqq y,\qquad \{x,y\}\in \overline T.
\end{equation*}
Note that both of those mappings are linear contractions if we consider them  as operators from the Hilbert space $\gil(T)$ into $\hil$ and $\kil$, respectively:
\begin{equation*}
    \pt\in \balg(\gil(T);\hil), \|\pt\|\leq 1\qquad\mbox{and}\qquad \qt\in \balg(\gil(T);\kil), \|\qt\|\leq 1.
\end{equation*}
Therefore, their adjoint operators $P_T^*\in\balg(\hil;\gil(T))$ and $Q_T^*\in\balg(\kil;\gil(T))$ are themselves linear contractions, and their ranges $\ran P_T^*$ and $\ran Q_T^*$ are linear relations.

Below we are going to examine the properties of the contractions $P_T$ and $Q_T$ and their connection with $T$ in detail. First   let us establish a few elementary facts. 
\begin{proposition}
Let $T$ be a linear relation between $\hil$ and $\kil$. Then
\begin{enumerate}[label=\textup{(\alph*)}]
    \item $\pt=P_{\overline T}$ and $\qt=Q_{\overline{T}}$,
    \item $T$ is (the graph of) a closable operator if and only if $\ker P_T$ is trivial.
\end{enumerate}  
\end{proposition}
\begin{proof}
The proof of (a) is straightforward from the definition of $P_T$. Statement (b) is obtained by noticing that $\ker\pt=\{0\}\times \mul \overline T$.
\end{proof}
In view of part (a) of the preceding proposition, there is no loss of generality in assuming that the linear relation $T$ is closed. In light of this, with a few exceptions, we will do so.

We start out by analysing the first coordinate projection $P_T$.
\begin{lemma}\label{L:lemma1}
Let $T$ be a linear relation between $\hil$ and $\kil$, then for every $h\in\hil$ we have $\qt P_T^*h\in\dom T^*$, that is, 
\begin{equation}
    \ran Q^{}_TP^*_T\subseteq \dom T^*.
\end{equation}
\end{lemma}
\begin{proof}
Consider $h\in \hil$ and let $P_T^*h\coloneqq \{z,w\}\in \ran P_T^*$, then we have 
\begin{equation*}
    \sip{x}{z}^{}_{\hil}+\sip{y}{w}^{}_{\kil}=\sipt[\big]{\{x,y\}}{P_T^*h}=\sip{x}{h}^{}_{\hil},
\end{equation*}
for every $\{x,y\}\in T$. Hence we get
\begin{equation*}
    \sip{y}{w}^{}_{\kil}=\sip{x}{h-z}^{}_{\hil},
\end{equation*}
which implies that $\{w,h-z\}\in T^*$ and therefore  $w=\qt P^*_Th\in\dom T^*$.
\end{proof}
\begin{proposition}\label{L:lemma2}
Let $T$ be a linear relation between $\hil$ and $\kil$, then
\begin{equation*}
    \dom T^*T=\pt(T\cap \ran P_T^*).
\end{equation*}
\end{proposition}
\begin{proof}
Assume first that $\{x,y\}\in T\cap \ran P^*$, then $k\in \ran Q^{}_TP_T^*\subseteq \dom T^*$ by Lemma \ref{L:lemma1}, and therefore there exists $z\in\hil$ such that $\{y,z\}\in T^*$. This means that $\{x,z\}\in T^*T$ and therefore $x\in\dom T^*T$.

Suppose on the converse that $x\in\dom T^*T$, then $\{x,y\}\in T$ and $\{y,z\}\in T^*$ for some $y$ and $z$. It suffices to show that $\{x,y\}\in \ran P^*$. Let therefore $\{u,v\}\in T$, then we have 
\begin{align*}
\sipt[\big]{\{x,y\}}{\{u,v\}}&=\sip{x}{u}^{}_{\hil} +\sip yv^{}_{\kil}=\sip{x}{u}^{}_{\hil}+\sip{z}{u}^{}_{\hil}\\ 
&= \sip[\big]{x+z}{P_T\{u,v\}}^{}_{\hil}=\sipt[\big]{P_T^*(x+z)}{\{u,v\}},  
\end{align*}
whence it follows that $\{x,y\}=P_T^*(x+z)\in\ran P_T^*$.
\end{proof}
From the above lemma we get the following two straightforward corollaries:
\begin{corollary}
If $T$ is (the graph of) an operator, then
\begin{equation*}
    T|^{}_{\dom T^*T}= T\cap \ran P_T^*. 
\end{equation*}
If $T$ is closed in addition, then
\begin{equation*}
    T|^{}_{\dom T^*T}=\ran P_T^*.
\end{equation*}
\end{corollary}
\begin{corollary}\label{C:cor3}
If $T$ is a closed linear relation then
\begin{equation*}
    \dom T^*T=\dom(\ran P_T^*).
\end{equation*}
\end{corollary}
Next we deal with the linear relation $\ran P_T^*\subseteq \overline{T}$. As it will turn out from the ensuing result, it is closely related to the regular part of $T$: 
\begin{theorem}\label{T:theorem1}
Let $T$ be a linear relation between $\hil$ and $\kil$, then  
\begin{equation*}
    \overline{\ran P_T^*}=\overline{\treg}.
\end{equation*}
In particular, $\ran P_T^*$ is always (the graph of) a closable operator.
\end{theorem}
\begin{proof}
First of all note that if $\{x,(I-P_m)y\}\in \overline T$ for some $\{x,y\}\in\hil\times \kil$, then necessarily $\{x,y\}\in\overline T$. Indeed, 
\begin{equation*}
    \{x,y\}=\{x,(I-P_m)y\}+\{0,P_my\}\in \overline{T}\,\widehat+\,(\{0\}\times \mul (\overline{T}))\subseteq \overline{T}+\overline{T}=\overline{T}.
\end{equation*}
Consequently, 
\begin{align*}
    \set[\big]{\{x,y\}\in \overline{T}}{y\in \ran P_m^{\perp}}&=\set[\big]{\{x,(I-P_m)y\}}{\{x,y\}\in \hil\times\kil, \{x,(I-P_m)y\}\in\overline{T}}\\
    &=\set[\big]{\{x,(I-P_m)y\}}{\{x,y\}\in\overline{T}, \{x,(I-P_m)y\}\in\overline{T}}.
\end{align*}
Since we have $\ker P_T=\{0\}\times \mul \overline T$, it follows that 
\begin{align*}
    \overline{\ran P_T^*}&=\big(\{0\}\times \mul\overline{T}\big)^\perp\\
    &=\set[\big]{\{x,y\}\in\overline{T}}{y\in(\mul\overline{T})^\perp }\\
    &=\set[\big]{\{x,y\}\in\overline{T}}{y\in\ran(I-P_m) }\\
    &=\set[\big]{\{x,(I-P_m)y\}}{y\in\kil, \{x,(I-P_m)y\}\in \overline{T} }\\
    &=\set[\big]{\{x,(I-P_m)y\}}{\{x,y\}\in\overline{T}}\\
    &=(\overline{T})_{\reg}=\overline{\treg},
\end{align*}
as it is claimed.
\end{proof}

We continue by describing the kernel and range spaces of the contractions $P_T$ and $P_T^*$:
\begin{theorem}
For every linear relation $T$ between $\hil$ and $\kil$ we have
\begin{enumerate}[label=\textup{(\alph*)}]
    \item $\ker P_T=\{0\}\times \mul \overline{T}$,
    \item $\ran P_T=\dom \overline{T}$,
    \item $\ker P_T^*=\mul T^*$,
    \item $\ran P_T^*=(\overline{\treg})|_{\dom T^*\overline{T}}.$
\end{enumerate}
\end{theorem}
\begin{proof}
Statements (a)-(c) are all straightforward, only point (d) needs some explanation. Note that we have identity $P_T=P_{\overline{T}}$ for every linear relation. On the other hand, $\overline{\treg}=(\overline{T})_{\reg}$ according to \eqref{E:tregclosure}. Therefore, without loss of generality we may assume that $T$ is closed, in which case (d) reduces to 
\begin{equation}\label{E:ranP=treg-2}
    \ran P_T^*=(\treg)|_{\dom T^*T}.
\end{equation}
By Theorem \ref{T:theorem1} we have $\ran P_T^*\subseteq \treg$ and by Corollary \ref{C:cor3}, $\dom (\ran P_T^*)=\dom T^*T$ whenever $T$ is closed. Since $\treg$ is the graph of an operator, we obtain
 \eqref{E:ranP=treg-2}. 
\end{proof}


\section{Linear relations adjoint to each other}
Let $T$ and $S$ be linear relations between $\hil$ and $\kil$, respectively, $\kil$ and $\hil$. We say that $T$ and $S$ are adjoint to each other (or that $T,S$ form an adjoint pair), if they satisfy
\begin{equation}\label{E:adjointpair1}
    T\subset S^*\qquad\mbox{and}\qquad S\subset T^*,
\end{equation}
or equivalently, if
\begin{equation}\label{E:adjointpair2}
    \sip yv^{}_{\kil}=\sip xu^{}_{\hil},\qquad \forall \{x,y\}\in T,\quad \forall \{v,u\}\in S.
\end{equation}
An important and natural question is under what conditions are the equations $T=S^*$ and $S=T^*$.
Below we provide  some necessary and sufficient conditions on the pair $S,T$  by means of the corresponding graph contractions $\pt,\qt$ and $P_S^{},Q_S^{}$ in order that they satisfy the weaker identities 
$T^{**}=S^*$ and $S^{**}=T^*$
 
\begin{theorem}\label{T:adjointpair}
Let $S,T$ be linear relations between $\hil$ and $\kil$, respectively, $\kil$ and $\hil$, that are adjoint to each other in the sense of \eqref{E:adjointpair1}. Then the  following statements are equivalent:
\begin{enumerate}[label=\textup{(\roman*)}]
    \item $S^*=T^{**}$ and $T^*=S^{**}$,
    \item \begin{enumerate}[label=\textup{(\alph*)}]
        \item $P_T^{}P_T^*+Q_S^{}Q_S^*=I_{\hil}$,
        \item $P_S^{}P_S^*+Q_T^{}Q_T^*=I_{\kil}$,
        \item $Q^{}_TP_T^*=P_S^{}Q_S^*$.
    \end{enumerate}
\end{enumerate}
\end{theorem}
\begin{proof}
Before proving the corresponding equivalences, let us introduce the following operator matrix
\begin{equation}\label{E:UTS}
    U_{T,S}\coloneqq \opmatrix{P_T}{-Q_S}{Q_T}{P_S}:\gil(T)\times \gil(S)\to \hil\times \kil,
\end{equation}
which acts between $\overline T\times \overline S$ and $\hil\times\kil$ by the correspondence
\begin{equation*}
    U_{T,S}\pair{\{x,y\}}{\{v,u\}}\coloneqq \{x-u,y+v\},\qquad \{x,y\}\in T, \{v,u\}\in S.
\end{equation*}
As $T$ and $S$ are adjoint to each other, one concludes that $U_{T,S}$ is an isometry: for let $\{x,y\}\in T$ and $ \{v,u\}\in S$, then by \eqref{E:adjointpair2}
\begin{align*}
     \left\|U_{T,S}\pair{\{x,y\}}{\{v,u\}} \right\|^2&=\|x-u\|^2+\|y+v\|^2\\
     &=\|x\|^2+\|y\|^2+\|v\|^2+\|u\|^2+2\re [\sip yv^{}_{\kil}-\sip xu^{}_{\hil}]\\
     &=\left\|\pair{\{x,y\}}{\{v,u\}} \right\|^2.
\end{align*}

Assume now (i). Since the corresponding canonical contractions of $T$ and $\overline T$ (resp., of $S$ and $\overline{S}$) are identical,  we may assume without loss of generality that both $S$ and $T$ are closed.  Denote by $W$ the `flip' operator \eqref{E:Wflip}. If we have $S^*=T$ and $T^*=S$, then the orthogonal decomposition
\begin{equation*}
    \hil\times\kil = T\,\widehat \oplus\, W(S)
\end{equation*}
implies that for every pair $\{h,k\}\in\hil\times \kil$ there exists $\{x,y\}\in T$ and $\{v,u\}\in S$ such that 
\begin{align*}
     \{h,k\}=\{x,y\}+\{-u,v\}= U_{T,S}\pair{\{x,y\}}{\{v,u\}},
\end{align*}
which means that $U_{S,T}$ is surjective, and hence a unitary operator. As a consequence we get $U^{}_{T,S}U_{T,S}^*=I^{}_{\hil\times\kil}$, that is, \begin{equation*}
    \opmatrix{I_\hil}{0}{0}{I_\kil}=\opmatrix{P_T^{}P_T^*+Q^{}_SQ_S^*}{P_T^{}Q_T^{*}-Q_S^{}P_S^{*}}{Q_T^{}P_T^{*}-P_S^{}Q_S^{*}}{P_S^{}P_S^*+Q_T^{}Q_T^*},
\end{equation*}
that clearly implies (ii). 

For the converse, assume (ii) and also that $S$, $T$ are closed. Since $S$ and $T$ are adjoint to each other, it suffices to show that $T^*\subset S$ and $S^*\subset T$. Consider a pair $\{w,z\}\in T^*$. By (ii) (a)-(c), we infer that $U_{T,S}$ is unitary and therefore we can find $\{x,y\}\in T$ and $\{v,u\}\in S$ such that 
\begin{equation*}
    \{-z,w\}=U_{T,S}\pair{\{x,y\}}{\{v,u\}}=\{x-u,y+v\}.
\end{equation*}
Here, 
\begin{equation*}
    \sip[\big]{\{x,y\}}{\{-u,v\}}=0=\sip[\big]{\{x,y\}}{\{-z,w\}},
\end{equation*}
and hence, 
\begin{equation*}
    \|\{x,y\}\|^2= \sip[\big]{\{x,y\}}{\{x-u,y+v\}}= \sip[\big]{\{x,y\}}{\{-z,w\}}.
\end{equation*}
This entails that $\{x,y\}=\{0,0\}$, and hence that $\{w,z\}=\{v,u\}\in S$. An analogous argument shows that $S^*\subset T$. 
\end{proof}
\begin{remark}
 Some characterizations of those linear operators $S,T$ which satisfy  identities $S^*=T$ and $T^*=S$ where given in \cite{SebTarcs2019} by means of the operator matrix 
 \begin{equation*}
     M_{T,S}\coloneqq \opmatrix{I_{\hil}}{-S}{T}{I_{\kil}},
 \end{equation*}
 cf. also \cites{SebTarcs2016,Popovici}. The general case of linear relations was discussed in \cite{Sandovici} in the same spirit. For an exact interpretation of matrices with linear relation entries the reader is referred to \cite{Hassi2020}.
\end{remark}
\section{Products of graph contractions}

Let $T$ be linear relation between the Hilbert spaces $\hil$ and $\kil$ and consider its  canonical graph contractions $\pt:\gil(T)\to \hil$ and $\qt:\gil(T)\to\kil$. Then the following four operators $ \pt P_T^*, \pt Q_T^*, \qt P_T^*$ and $ \qt Q_T^* $ are all well defined linear contractions between the appropriate Hilbert spaces. In this section we clarify their role and connection with the relations $T$ and $T^*$. 
\begin{lemma}\label{L:TT}
Let $T$ be a closed linear relation between $\hil$ and $\kil$, then
\begin{enumerate}[label=\textup{(\alph*)}]
    \item $\mul (TT^*)=\mul T$,
    \item  $T^*T=(\treg)^*\treg^{},$
    \item  $(TT^*)_{\reg}=\treg (T^*)_{\reg}$.
\end{enumerate}
\end{lemma}
\begin{proof}
(a) Let $k\in\mul T$, then $\{0,0\}\in T^*$ and $\{0,k\}\in T$ implies that $k\in \mul TT^*$. Assume on the converse that $k\in\mul TT^*$, then there exists $u$ such that  $\{0,u\}\in T^*$ and $\{u,k\}\in T^*$. Since we have $u\in \dom T^*\cap \mul T=\{0\}$, it follows that $\{0,k\}\in T$ and therefore that $k\in\mul T$. 

(b) First we show inclusion $T^*T\subset(\treg)^*\treg^{}$. Take $\{x,z\}\in T^*T$, then there exists $y$ such that  $\{x,y\}\in T$ and $\{y,z\}\in T^*$. In particular we have $y\in\dom T^*\subseteq \mul T^\perp$, thus $\{x,y\}=\{x,(I-P_m)y\}\in\treg$.
 On the other hand, we have inclusion $\treg\subset  T$ by closedness, so $T^*\subset(\treg)^*$. Consequently, $\{x,z\}\in (\treg)^*\treg$, indeed. Let now $\{x,z\}\in(\treg)^*\treg$, then there exists $y$ such that $\{x,y\}\in \treg$ and $\{y,z\}\in (\treg)^*$. Here we have $\{x,y\}\in T$ as $\treg\subseteq T$. Furhtermore, $(\treg)^*$ can be written as 
\begin{equation*}
    (\treg)^*=T^*\,\widehat\oplus\,(\mul T\times\{0\}),
\end{equation*}
where $\widehat\oplus$ denotes Minkowski direct sum. This yields us $\{k,h\}\in T^*$ and $w\in\mul T$ such that $\{y,z\}=\{k,h\}+\{w,0\}$. Since $y,k\in\mul T^\perp,$ we get $w=0$ and $y=k$, consequently $\{y,z\}\in T^*$ and $\{x,z\}\in T^*T$.

(c) By (a) we have $\mul T=\mul (TT^*)$, hence
\begin{equation*}
    (TT^*)_{\reg}=(I-P_m) TT^*=T_{\reg} T^*\supset T_{\reg} (T^*)_{\reg},
\end{equation*}
because $(T^*)_{\reg}\subset T^*$. To see the converse inclusion take $\{v,(I-P_m)w\}\in (TT^*)_{\reg}$, and let $\{v,u\}\in T^*$ and $\{u,w\}\in T$ for some $u$, then $\{u,(I-P_m)w\}\in \treg$ and from $u\in \dom T$ we get that $u\in (\mul  \overline{T})^{\perp}$, hence $\{v,u\}\in (T^*)_{\reg}$. Thus $\{v,(I-P_m)w\}\in T_{\reg} (T^*)_{\reg}.$
\end{proof}
 
In the next theorem we are going to deal with the contractions $\pt P_T^*, \qt Q_T^*, \pt Q_T^*$ and $\qt P_T^*$. 

\begin{theorem}\label{T:maintheorem}
Let $T$ be a closed linear relation between $\hil$ and $\kil$. Then 
\begin{enumerate}[label=\textup{(\alph*)}]
    \item $P^{}_TP_T^*=(I+T^*T)^{-1}$,
    \item $Q_T^{}P_T^*=\treg(I+T^*T)^{-1}$,
    \item $P_T^{}Q_T^{*}=(T^*)_{\reg}(I+TT^*)^{-1}$,
    \item $Q^{}_TQ_T^*=I-(I+TT^*)^{-1}=P_m+(TT^*)_{\reg}(I+TT^*)^{-1}$.
\end{enumerate}
\end{theorem}
\begin{proof}
(a) 
Let us introduce the linear operator
\begin{equation*}
    P_T^{\dagger}:\dom T\to \hil\times \kil, \qquad P_T^{\dagger}u\coloneqq \{u, \treg u\}.
\end{equation*}
Observe that $P_T^{\dagger}u\in \treg\subseteq T$ for every $u\in\dom T$, and that 
\begin{equation}\label{E:PdaggerP}
P_T^{\dagger}P^{}_T\{u,\treg u\}=\{u,\treg u\}.    
\end{equation}
Since $\ran P_T^*\subset \treg$ by Theorem \ref{T:theorem1}, from \eqref{E:PdaggerP} it follows that $P_T^\dagger P^{}_TP_T^*=P_T^*$. 

Let now $x\in\dom T$ and $h\in\hil$, then 
\begin{align*}
    \sip{x}{h}^{}_{\hil}&=\sip{P^{}_TP_T^{\dagger}x}{h}^{}_{\hil}=\sipt{P_T^{\dagger}x}{P_T^*h}=\sipt{P_T^\dagger x}{ P_T^\dagger P^{}_TP_T^*h}\\
    &=\sipt[\big]{\{x,\treg x\}}{\{P^{}_TP_T^*h,\treg P^{}_TP_T^*h\}}\\
    &=\sip{x}{P^{}_TP_T^*h}^{}_{\hil}+\sip{\treg x}{\treg P^{}_TP_T^*h}^{}_{\kil},
\end{align*}
consequently, 
\begin{equation*}
    \sip{\treg x}{\treg P^{}_TP_T^*h}^{}_{\kil}=\sip{x}{h-P_T^{}P_T^*h}^{}_{\hil}.
\end{equation*}
This implies that
\begin{equation*}
    P^{}_TP_T^*h\in\dom (\treg)^{*}\treg^{}\qquad\mbox{and}\qquad h=(I+(\treg)^{*}\treg^{})P_T^{}P_T^*h,
\end{equation*}
that is, $P_TP_T^*=(I+(T^*)_{\reg}\treg^{})^{-1}$. Since we have identity $T^*T=(T^*)_{\reg}\treg^{}$ by Lemma \ref{L:TT}, the proof of part (a) is complete.

(b) Take any vector $h\in\hil$. From (a) and equality $P_T^{\dagger}P^{}_TP_T^*=P_T^*$ we conclude that 
\begin{align*}
    Q^{}_TP_T^*h&=Q^{}_TP_T^{\dagger}P^{}_TP_T^*h=Q^{}_TP_T^{\dagger}(I+T^*T)^{-1}h\\
    &=Q_T\{(I+T^*T)^{-1}h,\treg(I+T^*T)^{-1}h\}\\
    &=\treg(I+T^*T)^{-1}h,
\end{align*}
whence we get identity (b).

(c) Replacing  $T$ by $T^*$ in (b), we obtain that  
\begin{equation*}
    Q^{}_{T^*}P_{T^*}^*=(T^*)_{\reg}(I+TT^*)^{-1}.
\end{equation*}
On the other hand, it follows from Theorem \ref{T:theorem1} (ii) (c) that $P^{}_TQ_T^*=Q_{T^*}^{}P_{T^*}^*$, hence the desired identity follows.

(d) First we note that 
\begin{equation*}
    (I-P_m)Q_T=\treg P_T,
\end{equation*}
because for $\{x,y\}\in T$, 
\begin{equation*}
(I-P_m)Q_T\{x,y\}=(I-P_m)y=\treg x=\treg P_T\{x,y\}.    
\end{equation*}
From this and (c) we get that 
\begin{align*}
    (I-P_m)Q^{}_TQ_T^{*}&=\treg P^{}_TQ_T^{*}=\treg (T^*)_{\reg}(I+TT^*)^{-1}\\
    &=(TT^*)_{\reg}(I+TT^*)^{-1}.
\end{align*}
On the other hand, we have
\begin{equation*}
    Q_T^{}Q_T^{*}=I-P_{T^*}^{}P_{T^*}^*,
\end{equation*}
by Theorem \ref{T:adjointpair} (ii) (b). Since $\ran P_{T^*}=\dom T^*\subset (\mul T)^{\perp}$,  we get 
\begin{equation*}
    P_mQ_T^{}Q_T^{*}=P_m-P_mP_{T^*}^{}P_{T^*}^*=P_m.
\end{equation*}
From the above identities we get 
\begin{equation*}
    Q_T^{}Q_T^{*}=P_mQ_T^{}Q_T^{*}+(I-P_m)Q_T^{}Q_T^{*}=P_m+(TT^*)_{\reg}(I+TT^*)^{-1},
\end{equation*}
which completes the proof.
\end{proof}

\begin{remark}
 We notice that $P_T^{\dagger}$ appearing in the proof of the preceding theorem is identical with the Moore-Penrose inverse of $P_T$, cf. \cite{BenIsrael} or \cite{Beutler}. We also remark that the proof might be  slightly simplified when $T$ is a closed operator. Namely, in that case we have $T=\treg$ and $P^{\dagger}_T=P_T^{-1}$.
\end{remark}
\begin{corollary}
Let $T$ be a densely defined closed linear operator between two Hilbert spaces. Then  
\begin{enumerate}[label=\textup{(\alph*)}]
    \item $P^{}_TP_T^*=(I+T^*T)^{-1}$,
    \item $Q_T^{}P_T^*=T(I+T^*T)^{-1}$,
    \item $P_T^{}Q_T^{*}=T^*(I+TT^*)^{-1}$,
    \item $Q^{}_TQ_T^*=TT^*(I+TT^*)^{-1}$.
\end{enumerate}
\end{corollary}
\begin{proof}
The proof is straightforward from Theorem \ref{T:maintheorem} by noticing that $P_m=0$ whenever $T$ is a densely defined closed operator. 
\end{proof}

In the next proposition we describe the kernel and range spaces of the contractions $\qt$ and $Q_T^*$.
\begin{proposition}
For every linear relation $T$ between $\hil$ and $\kil$ we have
\begin{enumerate}[label=\textup{(\alph*)}]
    \item $\ker Q_T=(\hil\times\{0\})\cap \overline{T}$,
    \item $\ran Q_T=\ran \overline T$,
    \item $\ker Q_T^*=(\ran T)^{\perp}$,
    \item $\ran Q_T^*=(\overline{\treg})|_{\ran T^*\cap \dom \overline{T}}\,\widehat\oplus(\{0\}\times \mul \overline T)$.
\end{enumerate}
\end{proposition}
\begin{proof}
Throughout the proof we shall assume, for sake of simplicity, that $T$ is  closed. Statement (a) follows immediatly from
\begin{align*}
    \ker Q_T=\set[\big]{\{x,y\}\in T}{y=0}=\set[\big]{\{x,0\}}{\{x,0\}\in T}.
\end{align*}
Statements (b) and (c) are straightforward form the very definition of $Q_T$. Finally, let us prove assertion (d).  According to Theorem \ref{T:maintheorem} (c) and (d) we have
\begin{align*}
    \ran Q_T^*&=\set[\big]{\{P^{}_TQ_T^*k,Q_T^{}Q_T^*k\}}{k\in \kil}\\
    &=\set[\big]{\{(T^*)_{\reg}(I+TT^*)^{-1}k,\treg(T^*)_{\reg}(I+TT^*)^{-1}k+P_mk}{k\in \kil}\\
    &=\set[\big]{\{(T^*)_{\reg}(I+TT^*)^{-1}k,(TT^*)_{\reg}(I+TT^*)^{-1}k\}}{k\in \kil}\widehat+(\{0\}\times \mul T),
\end{align*}
where in the last equality we used that 
\begin{equation*}
    \ker (I+TT^*)^{-1}=\mul TT^*=\mul T=\ran P_m.
\end{equation*}
We have on the other hand
\begin{align*}
    &\set[\big]{\{(T^*)_{\reg}(I+TT^*)^{-1}k,(TT^*)_{\reg}(I+TT^*)^{-1}k\}}{k\in \kil}=\\
    &\qquad\qquad=\set[\big]{\{(T^*)_{\reg}z,\treg(T^*)_{\reg} z\}}{z\in \dom \treg(T^*)_{\reg}}\\
    &\qquad \qquad=\treg|_{\ran (T^*)_{\reg}\cap\,\dom \treg}.
\end{align*}
Finally we note that $\dom \treg=\dom T$ and that $\mul T^*=\overline{\dom T}$ and therefore 
\begin{equation*}
    \ran (T^*)_{\reg}\cap\,\dom \treg=\ran T^*\cap\,\dom T.  
\end{equation*}
This together with the above observations yields identity (d).
\end{proof}

\begin{corollary}
Let $T$ be a closed linear operator between $\hil$ and $\kil$, then 
\begin{equation*}
    \ran Q_T^*=T|_{\ran T^*\cap\, \dom T}.
\end{equation*}
\end{corollary}
\begin{corollary}
Let $T$ be a closed linear relation between $\hil$ and $\kil$, then
\begin{equation*}
     \treg=\treg|_{\dom T^*T}\,\widehat +\, \treg|_{\ran T^*\cap\, \dom T}.
\end{equation*}
If $T$   is a closed operator, then 
\begin{equation*}
    T=T|_{\dom T^*T}\,\widehat +\, T|_{\ran T^*\cap\, \dom T}.
\end{equation*}
\end{corollary}
\begin{proof}
Recall that the operator matrix $U_{T,T^*}$ defined by \eqref{E:UTS} with $S\coloneqq T^*$ is an isometry hence, in particular one has $P_T^*P_T^{}+Q_T^*Q_T^{}=I_{\gil(T)}$. As a consequence, we have by \cite{FW}*{Theorem 2.2} that
\begin{align*}
    T&=\ran (P_T^*P_T^{}+Q_T^*Q_T^{})^{1/2}=\ran P_T^*\,\widehat + \ran Q_T^*\\
        &=\treg|_{\dom T^*T}\,\widehat + \,\treg|_{\ran T^*\cap\, \dom T}\,\widehat +\, (\{0\}\times \mul T).
\end{align*}
Now the desired identity follows since $\treg=T\,\widehat\ominus \,(\{0\}\times \mul T)$.
\end{proof}

\bigskip

We conclude the paper with an application of the results. Let $T$ be a closed linear relation between the Hilbert spaces $\hil$ and $\kil$ and denote by $E_T$ the orthogonal projection of $\hil\times \kil$ onto $T$. Then $E_T$ admits a matrix representation as an operator in $\hil\times \kil$:
\begin{equation*}
    E_T=\opmatrix{E_{11}}{E_{12}}{E_{21}}{E_{22}},
\end{equation*}
where the components $E_{ij}$ are bounded operators between the appropriate Hilbert spaces. Recall that $E_T$ is was called the characteristic projection of $T$ by Stone, who proved that the entries $E_{ij}$ may be expressed in terms of $T$ and $T^*$, provided that $T$ is a densely defined and closed operator (see \cite{Stone}*{Theorem 4}, cf. also \cite{Jorgensen}*{Theorem 3}):
\begin{equation*}
    E_T=\opmatrix{(T^*T+I)^{-1}}{T^*(TT^*+I)^{-1}}{T(T^*T+I)^{-1}}{TT^*(TT^*+I)^{-1}}.
\end{equation*}
In \cite{Hassi2007}*{Lemma 6.4}, the above result of Stone was extended to closed linear relations. In the ensuing theorem we are going we restate this general result as a straightforward consequence of Theorem \ref{T:maintheorem}:
\begin{theorem}
Let $T$ be a closed linear relation between the Hilbert spaces $\hil$ and $\kil$. Then the characteristic projection $E_T$ of $T$ can be written as 
\begin{equation*}
    E_T=\opmatrix{P_T^{}P_T^*}{P_T^{}Q_T^*}{Q_T^{}P_T^*}{Q_T^{}Q_T^{*}}=\opmatrix{(T^*T+I)^{-1}}{(T^*)_{\reg}(TT^*+I)^{-1}}{\treg(T^*T+I)^{-1}}{I-(TT^*+I)^{-1}}
\end{equation*}
\end{theorem}
\begin{proof}
Consider the canonical embedding operator $V_T:\gil(T)\to\hil\times \kil$, given by
\begin{equation*}
    V_T\coloneqq \pair{P_T}{Q_T}\{x,y\}\coloneqq \{x,y\},\qquad \{x,y\}\in T.
\end{equation*}
Clearly, $V_T$ is a linear isometry with range $T$ and therefore $V_T^{}V_T^{*}$ is identical with $E_T$, i.e., 
\begin{equation*}
    E_T=\pair{P_T}{Q_T}\begin{bmatrix}P_T^*&Q_T^*\end{bmatrix}=\opmatrix{P_T^{}P_T^*}{P_T^{}Q_T^*}{Q_T^{}P_T^*}{Q_T^{}Q_T^*}.
\end{equation*}
Theorem \ref{T:maintheorem} completes now the proof.
\end{proof}
\bigskip

\bibliographystyle{abbrv}

\end{document}